\documentclass[reqno]{amsart}
\usepackage{hyperref}

\begin{document}
\title[A non-local problem  \dots]
{A non-local problem for the fractional order Rayleigh-Stokes equation}

\author[Ravshan Ashurov, Oqila Mukhiddinova, Sabir Umarov,  \hfil \hfilneg] {Ravshan Ashurov, Oqila Mukhiddinova, Sabir Umarov}  % in alphabetical order

\address{Ravshan Ashurov \newline
V.I. Romanovskiy Institute of Mathematics,\\
Uzbekistan Academy of Science, \\
University str.,9, Olmazor district,\\
Tashkent, 100174, Uzbekistan}
\email{ashurovr@gmail.com}

\address{Oqila Mukhiddinova \newline
	Tashkent University of Information Technologies,\\
	108 Amir Temur Avenue, \\
	Tashkent, 100200, Uzbekistan \\
\\}
\email{oqila1992@mail.ru}

\address{Sabir Umarov \newline
University of New Haven,  \\
Department of Mathematics\\
300 Boston Post Road\\
West Haven, CT 06516,  USA\\
}
\email{sumarov@newhaven.edu}

\subjclass[2000]{} 
\begin{abstract}A nonlocal boundary value problem for the fractional version of the well known in fluid dynamics Rayleigh-Stokes equation is studied. Namely, the condition $u(x,T)=\beta u(x,0)+\varphi(x)$, where $\beta $ is an arbitrary real number, is proposed instead of the initial condition. If $\beta=0$, then we get the inverse problem in time, called the backward problem. It is well known that the backward problem is  ill-posed in the sense of Hadamard. If $\beta=1$, then the corresponding non-local problem becomes well-posed in the sense of Hadamard, and moreover, in this case a coercive estimate for the solution can be established.  The aim of this work is to find values of the parameter $\beta$, which separates two types of behavior of the semi-backward problem under consideration. We prove the following statements: if $\beta \ge 1,$ or $\beta<0$, then the problem is well-posed; if $\beta\in (0,1)$, then depending on the eigenvalues of the elliptic part of the equation, for the existence of a solution an additional condition on orthogonality of the right-hand side of the equation and the boundary function to some eigenfunctions of the corresponding elliptic operator may emerge.
\end{abstract}

\maketitle \numberwithin{equation}{section}
\newtheorem{theorem}{Theorem}[section]
\newtheorem{corollary}[theorem]{Corollary}
\newtheorem{lemma}[theorem]{Lemma}
\newtheorem{remark}[theorem]{Remark}
\newtheorem{problem}[theorem]{Problem}
\newtheorem{example}[theorem]{Example}
\newtheorem{definition}[theorem]{Definition}
\allowdisplaybreaks

\section{Introduction}

Fractional derivatives serve as an essential tools in modeling of complex processes. The concept of fractional derivatives arose simultaneously with derivatives of integer order. Starting with the work of Abel (see, e.g. \cite{KilSriTru}), the concept of fractional derivatives began to be widely used in various fields, such as electrochemistry, neuron models in biology, applied mathematics, fluid dynamics, viscoelasticity and fluid mechanics \cite{Kulish}. Models with fractional derivatives are used to analyze viscoelasticity, for example, of polymers during glass transition and in the glassy state \cite{Debnath}, the theoretical base of which is the well-known Rayleygh-Stokes equation. A fractional model of a generalized second-class fluid flow can be represented as the Rayleigh-Stokes problem with a time-fractional derivative \cite{Bazh}:
\begin{equation}\label{probIN}
	\left\{
	\begin{aligned}
		&\partial_t u(x,t)  -(1+\gamma\, \partial_t^\alpha)\Delta u(x,t) = f(x, t),\quad x\in \Omega, \quad 0< t \leq T;\\
		&u(x, t) = 0, \quad x\in \partial \Omega, \quad 0 < t \leq T;\\
		&u(x, 0)= \varphi(x), \quad x\in \Omega,
	\end{aligned}
	\right.
\end{equation}
where $1/\gamma>0$ is the fluid density, a fixed constant, the source term $f(x,t)$ and the initial data $\varphi$ are given functions, $\partial_t= \partial/\partial t$, and $\partial_t^\alpha$ is the Riemann-Liouville fractional derivative of order $\alpha \in (0,1)$ defined by (see, e.g. \cite{KilSriTru}):
\begin{equation}\label{RL}
	\partial_t^\alpha h(t)= \frac{d}{dt} \int\limits_0^t \omega_{1-\alpha}(t-s) h(s)ds, \quad \omega_{\alpha}(t)=\frac{t^{\alpha-1}}{\Gamma(\alpha)}.
\end{equation}
Here $\Gamma(\cdot)$ is Euler's gamma function. Usually problem (\ref{probIN}) is considered in the domain $\Omega \subset R^N$, $N=1,2,3$, and for $N>1$ it is assumed that the boundary $\partial \Omega$ of  $\Omega$ is sufficiently smooth.

When $\alpha=1$ the equation in (\ref{probIN})  is also called the Haller equation. This equation is a mathematical model of water movement in capillary-porous media, including the soil. In this case, $u$ is humidity in fractions of a unit, $x$ is a point inside the soil, $t$ is time  (see, for example, in \cite{Chud}, formulas (1.4) and (1.84) on
p. 137 and 158, \cite{Nakh1}, formula (9.6.4) on p. 255, and \cite{Nakh2}, formula (2.6.1) on p. 59).  See also works \cite{Tan1}, \cite{Tan2}, where on the base of modified Darcy's law for a viscoelastic fluid, the first Stokes problem was extended to the problem for an Oldroyd-B fluid in a porous half-space, and equation (\ref{probIN}) was obtained as a mathematical model. Recall that usually Stokes' first problem describes flows caused by a suddenly accelerated plate for homogeneous incompressible isotropic fluids with pressure-dependent viscosity.

The Rayleigh-Stokes problem (\ref{probIN}) plays an important role in the study of the behavior of some non-Newtonian fluids, as well. A non-Newtonian fluid is a fluid that has a constant viscosity independent of stress, i.e. does not obey Newton's law of viscosity. The fractional derivative $\partial_t^\alpha$ is used in the equation (\ref{probIN}) to describe the viscoelastic flow behavior (see, for example, \cite{Fet}, \cite{Shen}).

In recent years, the Rayleigh-Stokes problem (\ref{probIN}) has received much attention due to its importance for applications (see, for example, \cite{Tan1} - \cite{AshurovVaisova}). An overview of work in this direction can be found in Bazhlekova et al. \cite{Bazh} (see also \cite{AshurovVaisova}). Properties of the solution of this model were studied by a number of authors applying various methods; see e.g. \cite{Fet}, \cite{Shen}, \cite{Zhao}. The authors of the work by Bazhlekova et al. \cite{Bazh} proved the Sobolev regularity of the homogeneous Rayleigh-Stokes problem for both smooth and non-smooth initial data $\varphi(x)$, including $\varphi(x) \in L_2(\Omega)$.
A number of works are devoted to the development of efficient and accurate numerical algorithms for solving the problem (\ref{probIN}). A survey of works in this direction is contained in the above-mentioned paper \cite{Bazh}. See also recent articles \cite{Le}, \cite{Dai}, and references therein.

The study of the inverse problem of determining the right-hand side of the Rayleigh-Stokes equation is the subject of many works (see, for example, \cite{Tran1}, \cite{Tran2}, \cite{Duc} and the bibliography cited there). Since this inverse problem is ill-posed in the sense of Hadamard, various regularization methods are considered in the above works, as well as numerical methods for finding the right-hand side of the equation are proposed.
We note  also that the inverse problem of determining the right-hand side of the equation is also ill-posed for the subdiffusion equation (see, for example, \cite{Kirane_f_1}, \cite{Kirane_f_2}, \cite{AshMuk}).

If the initial condition $u(x, 0)= \varphi(x)$ in the problem (\ref{probIN}) is replaced by $u(x, T)= \varphi(x)$, then one gets the, so-called, \emph{backward problem}. This problem is not well-posed, i.e. a small change in $u(x,T)$ leads to a large change of the solution. In papers \cite{Luc1}, \cite{Luc2} (see also references therein) various regularization methods are proposed, accompanied by verification of these methods using numerical experiments. We emphasize that in these papers $N<4,$ and this is connected with the method used there. Namely, if the dimension of the space is less than four, then for the eigenvalues $\lambda_k$ of the Laplace operator with the Direchlet condition, the series
\[
\sum_k \lambda_k^{-2}
\]
converges.

Let us focus, in more detail, on the recently published work \cite{AshurovVaisova}. In this paper, along with other questions, problem (\ref{probIN})  is investigated by taking the non-local condition $u(x, T)=\beta u(x,0) + \varphi(x)$ instead of the initial condition. The authors considered only the cases $\beta=0$ and $\beta=1$: if $\beta=0$ then we have the backward problem (note that here the dimension $N$ is arbitrary). The authors proved that if $\beta=0$, then the solution exists and is unique, but there is no stability. If $\beta=1$, then the problem is well-posed in the sense of Hadamard, i.e. the unique solution exists, and the solution continuously depends on the initial data and on the right-hand side of the equation.

The question naturally arises: what happens if $\beta$ takes other values than $0$ and $1$? In the present paper we consider a more general non-local condition  $u(x, t_0)=\beta u(x,0) + \varphi(x)$, $t_0\in (0,  T]$ and  provide a definitive answer to this question. The
main results of the current work can be formulated as follows:

1)  If $\beta \ge 1$, or $\beta<0$,  then the problem is well-posed in the sense of Hadamard: the solution exists, it is unique and stable;

2) case $\beta=0$ is considered in \cite{AshurovVaisova}: in this case there is a unique solution, but it is not stable;

3) if $\beta\in (0,1)$, then the well-posedness of the problem depends on the location of the spectrum (i.e., the eigenvalues $\lambda_k$) of the Laplace operator with the Direchlet condition. If the  inequality $B_\alpha(\lambda_k,t_0)\neq \beta$ (the definition of this function is given in Section 3) holds for all $k=1,2,\cdots$, then the problem is well-posed in the sense of Hadamard. If $B_\alpha(\lambda_k,t_0)=\beta$ for some $k\in K_0$ (it is proved in the paper that the set $K_0$ contains only a finite number of points), then a necessary and sufficient condition for the existence of a solution is found. However, in this case there is no unique solution.

In what concerns  the non-local condition
\[
u(x, T)=\beta u(x,0) + \varphi(x),
\] in the variable $t,$
the corresponding problem with the parameter $\beta=1$ for the classical diffusion equation was first considered in \cite{Ashyr}, \cite{AshyrSob} and \cite{AshyrSob1}.  In \cite{AF2022} and \cite{AF_1_2022}, this problem with an arbitrary parameter $\beta$ was studied in detail for subdiffusion equations with Riemann-Liouville and Caputo derivatives correspondingly. In a recent paper \cite{Tran}, the authors considered the subdiffusion equation with the Caputo-Fabrizio derivative  on an $N$-dimensional torus with a non-local condition
\[
\varepsilon u(T)= u(0)+ \varphi.
\]
In these works the cases $\varepsilon=0$ and $\varepsilon>0$ are studied separately. The authors also studied solution limit at $\varepsilon \to 0$. Note, in this paper if $\varepsilon=0$, then we have the Cauchy problem, whereas in our case we have the backward problem.

The present paper consists of five sections. Section 2 provides precise formulations of the problems
studied in this paper. In Section 3, we introduce the standard Hilbert space of "smooth" functions
via the power of an elliptic operator
and give some well-known properties of the function $B_\alpha(\lambda, t)$ introduced in \cite{Bazh}.
Here we prove an important lemma used for solution of the non-local problem in the variable $t.$
Section 4 is devoted to the study of the
main non-local problem  with operator $A$ generalizing the Laplace operator.

\section{Problem formulations}

Let $H$ be a separable Hilbert space. Denote by $(\cdot, \cdot)$ be the inner product
and by $||\cdot||$   the norm in $H$. Consider an arbitrary unbounded positive self-adjoint operator $A$ with a dense domain in $H$.
We assume that $A$ has a complete in $H$ system of orthonormal
eigenvectors (eigenfunctions) $\{v_k\}$ and a countable set of positive
eigenvalues
\[
\lambda_k: \quad 0<\lambda_1\leq\lambda_2 \cdot\cdot\cdot\rightarrow +\infty.
\]
We also assume that the set $\{\lambda_k\}$ does not have a finite limit point.

For a vector-valued functions $h: \mathbb{R}_+\rightarrow H$, we define the Riemann-Liouville fractional derivative of order $0<\alpha< 1$ in the same way as (\ref{RL}) (see, e.g. \cite{Liz}).
Finally, let $C((a,b); H)$ denote
the set of functions $u(t)$ continuous in $t\in (a,b)$
with values in $H$.

Consider the following non-local problem for the abstract Rayleigh-Stokes equation
accepting the integral (in the definition of the fractional derivative) in the sense of Bochner:
\begin{equation}\label{probMain}
	\left\{
	\begin{aligned}
		&\partial_t u(t)  + (1+\gamma\, \partial_t^\alpha)A u(t) = f(t),\quad 0< t \leq T;\\
		&u(t_0)=\beta u(0)+ \varphi,
	\end{aligned}
	\right.
\end{equation}
where $\gamma>0$ and $t_0\in (0, T]$ are fixed constants, $\varphi\in H,$
$f(t) \in C((0,T]; H),$ and $\beta$  is an arbitrary real number. If $\beta=0,$ then this problem is called \textit{the backward problem}.

\begin{definition}\label{def}
	A function $u(t)\in C([0,T]; H)$  is called a solution of the non-local Rayleigh-Stokes problem (\ref{probMain}), if
	\[
	\partial_t u(t), \,  A u(t), \, \partial_t^\alpha Au(t)\in C((0,T); H),
	\]
	and it satisfies conditions
	(\ref{probMain}) for all $t \in (0,T].$
\end{definition}

\begin{remark}\label{laplace}
	As an example of the operator $A$ one can take, for example, the Laplace operator with the Dirichlet condition in an arbitrary $N$ (not only $\leq 3$)-dimensional bounded domain with a sufficiently smooth boundary.
	This operator has all the properties listed above.
\end{remark}

\section{Preliminaries}

For a given real number $ \tau $, we define the operator $A^{\tau}$ by
$$
A^\tau h= \sum\limits_{k=1}^\infty \lambda_k^\tau h_k v_k.
$$
Note, that the operator $A$ is positive and therefore $\lambda_k>0$ for all $k.$ Here and everywhere
for the vector $h\in H$, the symbol $h_k$ will denote the Fourier coefficients of this vector: $h_k=(h, v_k)$. The domain of definition of the operator
$A^\tau$ is
$$
D(A^\tau)=\{h\in H:  \sum\limits_{k=1}^\infty \lambda_k^{2\tau}
|h_k|^2 < \infty\}.
$$
For elements of $h, g \in D(A^\tau)$ we introduce the norm
\[
||h||^2_\tau=\sum\limits_{k=1}^\infty \lambda_k^{2\tau} |h_k|^2 = ||A^\tau h||^2,
\]
and the inner product
\[
(h,g)_{\tau} = \sum_{k=1}^{\infty} \lambda_k^{2\tau} h_k g_k = (A^{\tau}h,A^{\tau}g).
\]
With this  inner product, the linear-vector space $D(A^\tau)$ becomes a Hilbert space.
Further, let $B_\alpha(\lambda, t)$ be a solution of the following Cauchy problem
\begin{equation}\label{B}
	L y(t)\equiv	y'(t)+\lambda (1+\gamma \partial_t^\alpha)y(t)=0, \,\, t>0,\,\,\lambda>0,\,\, y(0)=1.
\end{equation}
The solution of
%such an equation is
this problem can be expressed in terms of the generalized Wright function (see e.g. A.A. Kilbas et. al \cite{KilSriTru}, Example 5.3, p. 289 and \cite{Pskhu}).
%But
The function $B_\alpha(\lambda, t)$
%, the solution of this Cauchy problem,
is studied in detail in  Bazhlekova, Jin, Lazarov, and Zhou \cite{Bazh}. See also  Luc, Tuan, Kirane, Thanh \cite{Luc1}, where
% very
important lower bounds are obtained.

% We also note the work of Pskhu \cite{Pskhu}, where a more general Cauchy problem than (\ref{B}) was studied. From the results of this paper one can obtain a representation for the solution of the Cauchy problem (\ref{B}), which is very convenient for further research.

The authors of \cite{Bazh}, in particular, proved the following lemma.

\begin{lemma}\label{Bazh} Let $B_\alpha(\lambda, t)$ be a solution of the Cauchy problem (\ref{B}). Then
	\begin{enumerate}
		\item
		$B_\alpha (\lambda, 0)=1,\,\, 0<B_\alpha (\lambda, t)<1,\,\, t>0$,
		\item$\partial_t B_\alpha (\lambda, t)<0, \,\, t\geq 0$,
		\item$\lambda B_\alpha (\lambda, t)< C \min \{t^{-1}, t^{\alpha-1}\}, \,\,t>0$,
		\item
		$\int\limits_0^T B_\alpha (\lambda, t) dt \leq \frac{1}{\lambda}, \,\, T>0.$
	\end{enumerate}
	
\end{lemma}

The  function $B_\alpha (\lambda, t)$ has the representation
\cite{Bazh}
\begin{equation}\label{BInt}
	B_\alpha (\lambda, t)=\int\limits_0^\infty e^{-rt} b_\alpha(\lambda, r) dr,
\end{equation}
where
\begin{equation}
	\label{add1}
	b_\alpha(\lambda, r)=\frac{\gamma}{\pi} \frac{\lambda r^\alpha \sin \alpha \pi}{(-r+\lambda\gamma r^\alpha \cos \alpha \pi +\lambda)^2+(\lambda \gamma r^\alpha \sin \alpha \pi )^2}.
\end{equation}

\begin{lemma} (\cite{Bazh,AshurovVaisova})
	\label{BazhEquation} The Cauchy problem
	\begin{equation}\label{BCauchyIN}
		y'(t)+\lambda (1+\gamma \partial_t^\alpha)y(t)=f(t), \,\, t>0,\,\,\lambda>0,\,\, y(0)=y_0,
	\end{equation}
	has a unique solution, which has a representation
	\begin{equation}\label{BCauchySolutionIN}
		y(t)=y_0 B_\alpha(\lambda, t) + \int\limits_0^t B_\alpha(\lambda, t-\tau) f(\tau) d \tau.	
	\end{equation}
\end{lemma}

We will also need an estimate obtained in \cite{AshurovVaisova}  for the derivative of the function $B_\alpha (\lambda, t).$
In view of the importance of this assertion for our further considerations, we present it with a brief proof.
\begin{lemma}\label{B't}
	There is a constant $C>0$, such that
	\[
	|\partial_t B_\alpha (\lambda, t)| \leq \frac{C}{\lambda\,t^{2-\alpha} },\,\, t>0.
	\]
\end{lemma}
\begin{proof}
	Differentiating the function $B_{\alpha}(\lambda, t)$ defined in  (\ref{BInt}), we have
	\[
	\partial_t B_\alpha (\lambda, t)=-\int\limits_0^\infty r e^{-rt} b_\alpha(\lambda, r) dr.
	\]
	Therefore, in accordance with the definition of $b_\alpha(\lambda, r)$ in \eqref{add1}
	\begin{align*}
		|\partial_t B_\alpha (\lambda, t)| &\leq \frac{\gamma}{\pi} \int\limits_0^\infty
		\frac{\lambda r^\alpha \sin \alpha \pi}{(\lambda \gamma r^\alpha \sin \alpha \pi )^2}\,r  e^{-rt}  dr \\
		&= \frac{1}{\gamma\pi\lambda \sin \alpha \pi } \int\limits_0^\infty r^{1-\alpha} e^{-rt} dr.
	\end{align*}
	Now the change of the variable $\tau=rt$ implies
	\begin{align*}
		|\partial_t B_\alpha (\lambda, t)| &\leq 	\frac{t^{\alpha-2}}{\gamma\pi\lambda \sin \alpha \pi } \int\limits_0^\infty \tau^{1-\alpha} e^{-\tau} d\tau \\
		&=  \frac{t^{\alpha-2}}{\gamma\pi\lambda \sin \alpha \pi } \Gamma(2-\alpha)= \frac{C}{\lambda\,t^{2-\alpha} }.
	\end{align*}
\end{proof}

In what follows, $\lambda$ will be replaced by eigenvalues $\lambda_k$ of the operator $A$. The following important lower bound for $B_\alpha (\lambda_k, t)$ was obtained in  Luc, N.H., Tuan,
N.H., Kirane, M., Thanh, D.D.X \cite{Luc1}.
\begin{lemma}\label{BLower}The following estimate holds for all $t \in [0, T]$ and $k\geq 1$:
	\[
	B_\alpha (\lambda_k, t)\geq \frac{C(\alpha, \gamma, \lambda_1)}{\lambda_k},
	\]
	where
	\[
	C(\alpha, \gamma, \lambda_1)=\frac{\gamma \sin \alpha \pi}{4}\int\limits_0^\infty \frac{r^\alpha e^{-rT}}{\frac{r^2}{\lambda_1^2}+ \gamma^2 r^{2\alpha}+1} dr.
	\]
\end{lemma}

Next, we estimate the derivative $\partial_\lambda B_\alpha (\lambda, t_0)$ from above when $\lambda\geq \lambda_1>0$, where $\lambda_1$ is the first eigenvalue of the operator $A$.

\begin{lemma}
	\label{Main}
	Let $0<t_0\leq T$, $\gamma>0$, $\alpha\in (0,1)$ be given numbers. There exists a positive number $\Lambda_0=\Lambda_0(t_0, \gamma, \alpha, \lambda_1)>0$ such that for any $\lambda\geq \Lambda_0$ the inequality
	\begin{equation}\label{main}
		\partial_\lambda B_\alpha (\lambda, t_0)<0
	\end{equation}
	holds.
\end{lemma}

\begin{proof} We rewrite the function  $B_\alpha (\lambda, t_0)$  in the form
	\[
	B_\alpha (\lambda, t_0) = \frac{1}{\lambda}\int\limits_0^\infty e^{-t_0 r}b_{\alpha,1}(\lambda, r ) dr,	
	\]
	where
	\[
	b_{\alpha,1}(\lambda, r) =\frac{\gamma}{\pi} \,\,\frac{r^\alpha \sin\alpha \pi}{(-\frac{r}{\lambda} +\gamma r^\alpha \cos\alpha \pi +1)^2+(\gamma r^\alpha  \sin \alpha \pi)^2},
	\]
	Now differentiating with respect to $\lambda$ we have
	\begin{align}
		\partial_\lambda B_\alpha (\lambda, t_0) &=  -\frac{1}{\lambda^2}\int\limits_0^\infty e^{-t_0 r}b_{\alpha,1}(\lambda, r ) dr  \notag \\
		&+ \label{B'}
		\frac{2}{\lambda^3}\int\limits_0^\infty e^{-t_0 r}b_{\alpha,1}(\lambda, r )\,\,\frac{r\big[ \frac{r}{\lambda}-\gamma r^{\alpha} \cos\alpha \pi-1\big]}{(-\frac{r}{\lambda} +\gamma r^\alpha \cos\alpha \pi +1)^2+(\gamma r^\alpha  \sin \alpha \pi)^2} dr.
	\end{align}
	We estimate each term in the latter separately. For the first integral, taking into account the inequality
	$(a+b+c)^2\leq 3(a^2+b^2+c^2)$, we have
	\begin{align}
		(-\frac{r}{\lambda} +\gamma r^\alpha \cos\alpha \pi +1)^2+(\gamma r^\alpha  \sin \alpha \pi)^2 &\leq 3(\frac{r^2}{\lambda_1^2}+ \gamma^2 r^{2\alpha}+1)+\gamma^2 r^{2\alpha} \notag \\
		&\leq	
		\label{denominator}
		\left\{
		\begin{aligned}
			&c,\,\, r<1; \\
			&c\, r^2, \,\, r\geq 1.
		\end{aligned}
		\right.
	\end{align}
	We note that here and elsewhere in this section, $c$ denotes a constant (not necessarily the same one), depending on the fixed parameters $\lambda_1$, $\alpha$ and $\gamma$.
	Therefore
	\begin{equation}
		\label{I_1}
		\int\limits_0^1 e^{-t_0 r}b_{\alpha,1}(\lambda, r ) dr\geq c\int\limits_0^1 r^\alpha e^{-rt_0} dr
		=c\,t_0^{-1-\alpha}\, I_1(t_0),
	\end{equation}
	where
	\begin{equation}
		\label{I1}
		I_1(t_0)= \int\limits_0^{t_0} \xi^\alpha e^{-\xi} d\xi,
	\end{equation}
	and
	\begin{equation}
		\label{I_2}
		\int\limits_1^\infty e^{-t_0 r}b_{\alpha,1}(\lambda, r ) dr\geq c\int\limits_1^\infty r^{\alpha-2} e^{-rt_0} dr
		=c\,t_0^{1-\alpha}\, I_2(t_0).
	\end{equation}
	where
	\begin{equation}
		\label{I2}
		I_2(t_0)= \int\limits_{t_0}^\infty \xi^{\alpha-2} e^{-\xi} d\xi.
	\end{equation}
	It follows from  estimates \eqref{I_1} and \eqref{I_2} that
	\begin{equation}\label{firsttirm}
		-\frac{1}{\lambda^2}\int\limits_0^\infty e^{-t_0 r}b_{\alpha,1}(\lambda, r ) dr\leq -\frac{C_0(t_0)}{\lambda^2},
	\end{equation}
	where $C_0(t_0)=c \,\left(t_0^{-1-\alpha}\, I_1(t_0)+t_0^{1-\alpha}\, I_2(t_0)\right)$.
	
	Let us now estimate the second term in (\ref{B'}) from above. Denote
	\[
	R(r)= -\frac{r}{\lambda} +\gamma r^\alpha \cos\alpha \pi +1.
	\]
	Then $R(0)=1$ and we can choose a positive number $r_0$ such that for all $r\in (0,r_0)$ one has $R(r)\geq \varepsilon$ with some $\varepsilon>0$.
	For these $r$ we have
	\[
	b_{\alpha, 1} (\lambda, r)\leq  \frac{\gamma}{\pi}\,\,\frac{r^\alpha \sin\alpha \pi}{(-\frac{r}{\lambda} +\gamma r^\alpha \cos\alpha \pi +1)^2}\leq \frac{\gamma}{\varepsilon^2\pi}\,\,r^\alpha \sin\alpha \pi.
	\]
	Taking this into account, we obtain
	\begin{align}
		\int\limits_0^{r_0} & e^{-t_0 r}b_{\alpha,1}(\lambda, r )\,\,\frac{r\big[ \frac{r}{\lambda}-\gamma r^{\alpha} \cos\alpha \pi-1\big]}{(-\frac{r}{\lambda} +\gamma r^\alpha \cos\alpha \pi +1)^2+(\gamma r^\alpha  \sin \alpha \pi)^2} dr \notag \\	
		&\leq
		\frac{\gamma}{\varepsilon^3\pi}\sin\alpha \pi\int\limits_0^{r_0} e^{-t_0 r} r^{\alpha+1} dr \notag \\
		&= \frac{c}{t_0^{\alpha+2}}\int\limits_0^{r_0t_0} e^{-\xi} \xi^{\alpha+1} dr \leq \frac{c \,\Gamma (\alpha+2)}{t_0^{\alpha+2}}.
		\label{secondfirst}
	\end{align}
	Now let $r\geq r_0$. Then, evidently,
	\[
	b_{\alpha, 1} (\lambda, r)\leq  \frac{1}{\gamma\pi}\,\,\frac{1}{r^\alpha \sin\alpha \pi},
	\]
	and, moreover
	\begin{align*}
		& \frac{r\big| \frac{r}{\lambda}-\gamma r^{\alpha} \cos\alpha \pi-1\big|}{(-\frac{r}{\lambda} +\gamma r^\alpha \cos\alpha \pi +1)^2+(\gamma r^\alpha  \sin \alpha \pi)^2}	\\
		& \leq \frac{r^{1-2\alpha}}{(\gamma \sin\alpha \pi)^2}\,\big| \frac{r}{\lambda}-\gamma r^{\alpha} \cos\alpha \pi-1\big| \leq
		\frac{C r^{2-2\alpha}}{(\gamma \sin\alpha \pi)^2},
	\end{align*}
	where $C= \lambda_1^{-1}+r_0^{\alpha-1}\gamma^{-1}+r_0^{-1}$.
	Hence
	\begin{align*}
		\int\limits_{r_0}^\infty &e^{-t_0 r}b_{\alpha,1}(\lambda, r )\,\,\frac{r\big| \frac{r}{\lambda}-\gamma r^{\alpha} \cos\alpha \pi-1\big|}{(-\frac{r}{\lambda} +\gamma r^\alpha \cos\alpha \pi +1)^2+(\gamma r^\alpha  \sin \alpha \pi)^2} dr \notag \\	
		&\leq \frac{C}{\pi(\gamma \sin\alpha \pi)^3}\,\,\int\limits_{r_0}^\infty e^{-t_0 r}r^{2-3\alpha}dr\leq \frac{C}{\pi(\gamma \, r_0^\alpha \sin\alpha \pi)^3}\,\,\int\limits_{0}^\infty e^{-t_0 r}r^{2}dr \notag \\
		&=\frac{C \Gamma(3)}{\pi(\gamma \, r_0^\alpha\sin\alpha \pi)^3} \,\cdot \frac{1}{t_0^{3}}.
	\end{align*}
	It follows from the latter and estimate (\ref{secondfirst})) that the second term in (\ref{B'}) is estimated from above by the quantity
	\begin{equation}\label{secondtirm}
		\frac{c}{\lambda^3}\cdot \left[ \frac{1}{t_0^{\alpha+2}} +\frac{1}{t_0^{3}}\right].
	\end{equation}
	Finally estimates (\ref{firsttirm}) and (\ref{secondtirm}) imply
	\begin{equation}\label{B'lambda}
		\partial_\lambda B_\alpha (\lambda, t_0) \leq -\frac{1}{\lambda^{2}} \left[C_0(t_0) -\frac{c}{\lambda}\left(\frac{1}{t_0^{\alpha+2}}+\frac{1}{t_0^{3}}\right)\right] < 0.
	\end{equation}
	The lemma is proved.
\end{proof}

Let $\{\lambda_k\}$ be the set of eigenvalues of the operator $A$.
Recall that this set has no finite limit points. In particular, the multiplicity of any eigenvalue is finite. Let $\beta\in (0,1)$.
In our further analysis of non-local problem \eqref{probMain} we will face with the solution of the following equation
\begin{equation}\label{eqbeta}
	B_\alpha (\lambda, t_0)=\beta
\end{equation}
with respect to $\lambda$. If $\lambda_0$ is a root of equation (\ref{eqbeta}), then the set of all $k$ for which $\lambda_k=\lambda_0$ will be denoted by $K_0$. If there is no an eigenvalue $\lambda_k$ equal to $\lambda_0$, then evidently the set $K_0$ is empty.

\begin{remark}\label{K0}
	According to Lemma \ref{Main}, starting from some number $k$, the function $B_\alpha (\lambda_k, t_0)$ decreases strictly with respect to $\lambda_k$ (if the multiplicity of the eigenvalue $\lambda_k$ is not taken into account). Therefore, the set $K_0$ is always finite.
\end{remark}

Thus, Lemma \ref{Main} states that only a finite number of eigenvalues $\lambda_k$ can be solutions of equation (\ref{eqbeta}). It can also be proved that if $t_0$ is large enough, then there can be only one such an eigenvalue. Indeed, the following statement is true:

\begin{lemma}\label{MainT0} Let $\gamma>0$ and $\alpha\in (0,1)$ be given numbers. There exists a positive number $T_0=T_0(\gamma, \alpha, \lambda_1)\geq 1$ such that for any $t_0$ in the interval $T_0\leq t_0\leq T$ the inequality
	\begin{equation}\label{mainT0}
		\partial_\lambda B_\alpha (\lambda, t_0)<0, \quad \lambda\geq \lambda_1.
	\end{equation}
	holds.
\end{lemma}

\begin{proof}It is easy to see that in estimates (\ref{firsttirm}) and (\ref{secondtirm}) the parameter $\lambda$ can be replaced by its smallest value $\lambda_1$ and the corresponding estimates remain valid. Then estimate (\ref{B'lambda}) takes the form:
	\begin{equation}\label{B'lambda_t}
		\partial_\lambda B_\alpha (\lambda, t_0) \leq -\frac{1}{\lambda_1^{2}} \left[C_0(t_0) -\frac{c}{\lambda_1}\left(\frac{1}{t_0^{\alpha+2}}+\frac{1}{t_0^{3}}\right)\right].
	\end{equation}
	Now suppose that $t_0 \geq 1$ and estimate $C_0(t_0)$ from below. For the integral $I_1(t_0),$ defined in \eqref{I1}, we have
	\[
	I_1(t_0)\geq \int\limits_0^{1} \xi^\alpha e^{-\xi} d\xi\geq \frac{1}{(\alpha+1)\, e}.
	\]
	Similarly for the integral $I_2(t_0)$, defined in \eqref{I2}, after integration by parts twice, we obtain
	\begin{align*}
		I_2(t_0) &=\frac{t_0^{\alpha-1} e^{-t_0}}{1-\alpha}+\frac{t_0^{\alpha} e^{-t_0}}{(1-\alpha)\alpha}+\frac{1}{(1-\alpha)\alpha}\int\limits_{t_0}^\infty \xi^{\alpha} e^{-\xi} d\xi \\
		&\geq e^{-t_0}\left[\frac{t_0^{\alpha-1} }{1-\alpha}+\frac{t_0^{\alpha} }{(1-\alpha)\alpha}+\frac{t_0^{\alpha+1} }{(1-\alpha)\alpha(\alpha+1)}\right].
	\end{align*}
	Therefore, for sufficiently large $t_0$ we have	
	\[
	C_0(t_0)\geq c\, t_0^{-1-\alpha}.
	\]
	Consequently, estimate (\ref{B'lambda_t}) takes the form
	\[
	\partial_\lambda B_\alpha (\lambda, t_0) \leq -\frac{c}{\lambda_1^{2}\, t_0^{1+\alpha}} \left[1 -\frac{1}{\lambda_1}\left(\frac{1}{t_0}+\frac{1}{t_0^{2-\alpha}}\right)\right].
	\]
	This implies the assertion of the lemma.	
\end{proof}

\begin{remark}
	\label{}
	Under the conditions of this lemma, only one eigenvalue $\lambda_{k_0}$ may satisfy equation (\ref{eqbeta}). Let the multiplicity of $\lambda_{k_0}$ be equal to $p$. Then $K_0=\{k_0, k_0+1, \cdots, k_0+p-1\}$.
	
	We also note that in Lemma \ref{Main} $\lambda$ is a sufficiently large and $t_0$ is an arbitrary positive number, in Lemma \ref{MainT0}, on the contrary, $t_0$ is sufficiently large and $\lambda\geq \lambda_1$  is an arbitrary number.
	
\end{remark}

\section{Existence of a solution of the non-local problem \eqref{probMain}}

To solve the non-local problem (\ref{probMain}), we divide it into two auxiliary
problems:
\begin{equation}\label{prob1a}
	\left\{
	\begin{aligned}
		&\partial_t \omega(t)  + (1+\gamma\, \partial_t^\alpha)A \omega(t) = f(t),\quad 0< t \leq T,\\
		&\omega(0) =0,
	\end{aligned}
	\right.
\end{equation}
and
\begin{equation}\label{prob1b}
	\left\{
	\begin{aligned}
		&\partial_t w(t) + (1+\gamma\, \partial_t^\alpha)A w(t) = 0,\quad 0< t \leq T,\\
		&w(t_0) = \beta w(0)+\psi,
	\end{aligned}
	\right.
\end{equation}
where $\psi\in H$ is a given element and $t_0$ is any fixed number from $(0, T]$.

Problems (\ref{prob1a}) and (\ref{prob1b}) are a special cases of problem (\ref{probMain}), so solutions to problems (\ref{prob1a}) and (\ref{prob1b}) are defined similarly to Definition \ref{def}.

\begin{lemma}\label{1=2+3}
	Let $\psi$ in \eqref{prob1b} have the form $\psi=\varphi-\omega(t_0),$ where $\varphi$ is a function in  non-local problem (\ref{probMain}). If $\omega(t)$ and $w(t)$
	are solutions of problems (\ref{prob1a}) and (\ref{prob1b}) correspondingly, then the solution to problem (\ref{probMain}) has the form $u(t)=\omega(t)+w(t)$ .
\end{lemma}

Proof of this lemma is standard, and therefore, we omit it.

Auxiliary problem (\ref{prob1a}) is solved in \cite{AshurovVaisova}. Let us formulate the corresponding result:

\begin{theorem}\label{Thprob1a}
	Let $\varepsilon\in (0,1)$ and $f(t) \in C ([0,T];D(A^\varepsilon))$. Then problem (\ref{prob1a}) has a unique solution
	\begin{equation}\label{v_fi0}
		\omega(t)= \sum\limits_{k=1}^\infty
		\left[\int\limits_{0}^tB_\alpha (\lambda_k, t-\tau)f_k(\tau)d\tau\right] v_k.
	\end{equation}
	Moreover, there is a constant $ C_\varepsilon> 0 $, such that
	\begin{equation}\label{estimatev_fi0}
		||\partial_t \omega(t)||^2 + ||\partial_t^\alpha A \omega(t)||^2 \leq C_\varepsilon
		\max\limits_{t\in[0,T]} ||f||_\varepsilon^2,\quad 0 \leq t \leq T.
	\end{equation}
\end{theorem}

Now let us consider non-local problem (\ref{prob1b}). We will seek the solution of this problem in the form of a generalized Fourier series
\[
w(t) = \sum\limits_{k=1}^\infty T_k(t) v_k,
\]
where $v_k$ are the eigenvectors of the operator $A$ and $T_k(t)$, $k\geq 1$, are  solutions of the following non-local
problems:
\begin{equation}\label{Teq}
	\left\{
	\begin{aligned}
		&	T'_k(t)+\lambda_k (1+\gamma \partial_t^\alpha)T_k(t)=0,\quad 0 < t \leq T;\\
		&T_k(t_0)=\beta T_k(0)+\psi_k,
	\end{aligned}
	\right.
\end{equation}
where $k \ge 1,$ $t_0\in (0, T]$ is a fixed point and $\psi_k$ is the Fourier coefficients of the element $\psi\in H$.
Denote $h_k=T_k(0), k=1,2,\dots.$
Then the unique solution to problem (\ref{Teq}) has the form (see Lemma \ref{BazhEquation})
\begin{equation}\label{Tk}
	T_k(t)=h_k B_\alpha (\lambda_k, t).
\end{equation}
To find the unknown numbers  $h_k,$ we use  the non-local conditions of (\ref{Teq}. Namely,
\[
h_k B_\alpha (\lambda_k, t_0)=\beta h_k + \psi_k.
\]
or, the same,
\begin{equation}\label{bk}
	h_k( B_\alpha (\lambda_k, t_0)-\beta)=\psi_k.
\end{equation}
If $\beta \geq 1,$  or $\beta < 0$ (note, $t_0>0$ and $\lambda_k>0$), then $B_\alpha (\lambda_k, t_0)\neq\beta$ due to Lemma \ref{Bazh}.
Therefore,  in these cases it follows from (\ref{bk}) that
\begin{equation}\label{bkEs1}
	h_k =\frac{\psi_k}{B_\alpha (\lambda_k, t_0)-\beta},
\end{equation}
and
\begin{equation} \label{bkEs12}
	|h_k|\leq C_\beta |\psi_k|, \,\, k\geq 1,
\end{equation}
where $C_\beta$ is a constant depending on $\beta$.
If $\beta = 0$, then $B_\alpha (\lambda_k, t_0)\neq 0$,  however,
in accordance with Lemma \ref{Bazh},  the function $B_\alpha (\lambda_k, t_0)$
asymptotically tends to zero, as $k \to \infty.$ Therefore, by Lemma \ref{BLower} one has:
$$
h_k =\frac{\psi_k}{B_\alpha (\lambda_k, t_0)},\,\, C_1 \lambda_k  |\psi_k|\leq |h_k|\leq
C_2 \lambda_k  |\psi_k|.
$$
Here constants $C_j, j=1,2, $ may depend on $\alpha, \gamma, \lambda_1$ and $t_0$. As noted above, this case has been studied in detail in
\cite{AshurovVaisova}. Therefore, we will not consider it further.

Now, let $0<\beta<1$ and consider equation (\ref{eqbeta}).
In accordance with Remark \ref{K0}, there are two possible cases: the set $K_0$ is empty and it
is not empty. If $K_0$ is empty, then since the set $\{\lambda_k\}$ does not have a finite limit point, the estimate in (\ref{bkEs1}) holds with some constant $C_\beta>0$ for all $k$.

Thus, if $\beta \in (0,1)$ and $K_0$ is empty, then the formal solution of problem (\ref{prob1b}) still has the form
\begin{equation}\label{w>1}
	w(t)=\sum\limits_{k=1}^{\infty} \frac{\psi_k}{B_\alpha (\lambda_k, t_0)-\beta} B_\alpha (\lambda_k, t)v_k.
\end{equation}

Finally, let us assume that $0<\beta<1$ and the set $K_0$ is not empty. In this case, due to equation (\ref{bk}), non-local problem (\ref{Teq}) has a solution if and only if the following orthogonality conditions are
verified:
\begin{equation}\label{orto}
	\psi_k=(\psi,v_k)=0, \,\, k\in K_0.
\end{equation}
Moreover, for  the values $k\in K_0$ arbitrary numbers $h_k$ are solutions of
equation (\ref{bk}). For all other $k$ we have
\begin{equation}\label{bkEs2}
	h_k =\frac{\psi_k}{B_\alpha (\lambda_k, t_0)-\beta},\quad
	|h_k|\leq C_\beta |\psi_k|, \quad k\notin K_0.
\end{equation}
Thus, the formal solution of problem (\ref{prob1b}) in this case
has the form
\begin{equation}\label{U1}
	w(t)=\sum\limits_{k\notin K_0} \frac{\psi_k}{B_\alpha (\lambda_k, t_0)-\beta} B_\alpha (\lambda_k, t)v_k+\sum\limits_{k\in K_0} h_k B_\alpha (\lambda_k, t)v_k.
\end{equation}

Now let us show
that the series (\ref{w>1}) and (\ref{U1}) indeed  define solutions to  non-local problem (\ref{prob1b}).
According to Definition \ref{def}, it suffices to show the applicability of the operators $\partial_t$ and $\partial_t^\alpha A$ term by term to these series and $w(t)\in C([0,T]; H)$, $\partial_t w(t), Aw(t), \partial_t^\alpha Aw(t)\in C((0,T); H)$. We
demonstrate this with the solution in (\ref{w>1}). In what concerns the solution in (\ref{U1}), then it is treated in exactly the same way.

Let $S_j(t), \ j \ge 1,$ be the sequence of partial sums of series (\ref{w>1}). Applying  Parseval's equality,
estimate (\ref{bkEs1}), and the first assertion of Lemma \ref{Bazh}, we have
\[
||S_j(t)||^2=\sum\limits_{k=1}^j \bigg|
\frac{\psi_k}{B_\alpha (\lambda_k, t_0)-\beta} B_\alpha (\lambda_k, t)\bigg|^2\leq C_\beta ||\psi||^2,
\]
Letting $j \to \infty,$ it follows from the latter that $w(t)\in C([0,T]; H)$. Further, we have
$$
A S_j(t) = \sum\limits_{k=1}^j  \frac{\lambda_k \psi_k}{B_\alpha (\lambda_k, t_0)-\beta}  B_\alpha (\lambda_k, t) v_k.
$$
Using the same reasoning (using the third assertion of Lemma \ref{Bazh}) as above, we get
\begin{equation}\label{Aw}
	||AS_j(t)||^2=\sum\limits_{k=1}^j \bigg|
	\frac{\lambda_k\psi_k}{B_\alpha (\lambda_k, t_0)-\beta} B_\alpha (\lambda_k, t)\bigg|^2\leq \frac{C_\beta}{t^{2(1-\alpha)}} ||\psi||^2,
\end{equation}
which implies that $Aw(t)\in C((0,T); H)$.
Next, applying Lemma \ref{B't}, we
have the following estimate:
\begin{equation}\label{Dalphaw}
	||\partial_t S_j(t)||^2 = \sum\limits_{k=1}^j  \left|\frac{\psi_k}{B_\alpha (\lambda_k, t_0)-\beta} \partial_t B_\alpha (\lambda_k, t)\right|^2\leq \frac{C_\beta}{\lambda_1t^{2(2-\alpha)}} ||\psi||^2.
\end{equation}
The latter implies $\partial_t w(t)\in C((0,T];H)$.
Equation (\ref{prob1b})written in the form $\partial_t w(t) = -Aw(t)-\gamma\, \partial_t^\alpha A w(t)$, $t>0$,
and the  estimates obtained above imply
\begin{equation}\label{Dw}
	||\partial_t^\alpha A w(t)||^2\leq \frac{C_\beta}{t^{2(2-\alpha)}} ||\psi||^2,
\end{equation}
Hence, $\partial_t^\alpha A w(t)\in C((0,T];H),$ as well.

Thus,  if $\beta \notin (0,1)$ or $\beta \in (0,1)$, but $K_0$ is an empty set, then the function (\ref{w>1}) is indeed a solution to problem (\ref{prob1b}).

Finally we show the uniqueness of the solution to problem
(\ref{prob1b}). Suppose we have two solutions: $w_1 (t)$, $w_2
(t)$ and set $w(t) = w_1 (t)- w_2 (t)$. Then we have
\[
\left\{
\begin{aligned}
	&\partial_t w(t) + (1+\gamma\, \partial_t^\alpha)A w(t) = 0,\quad 0< t \leq T;\\
	&w(t_0) = \beta w(0).
\end{aligned}
\right.
\]
Let $w(t)$ be any solution to this problem and let $w_k(t)=(w(t),v_k)$. Since the operator $A$ is self-adjoint,
one has
\begin{align*}
	\partial_t^\rho w_k(t) &=(\partial_t^\rho w(t),v_k)= - ( Aw(t),v_k )-\gamma(\partial_t^\alpha w, v_k)
	\\
	&= -\lambda_k(1+\gamma \partial_t^\alpha) w_k(t)
\end{align*}
or
\begin{equation}\label{wkeq}
	\partial_t^\rho w_k(t)+
	\lambda_k(1+\gamma \partial_t^\alpha) w_k(t)=0.
\end{equation}
It follows from the nonlocal condition  that
\begin{equation}\label{wkc}
	w_k(t_0) =\beta w_k(0).
\end{equation}

Let us denote $w_k(0)= h_k$. Then the unique solution to the
differential equation (\ref{wkeq}) with this initial condition has
the form $w_k(t)=h_k B_\alpha (\lambda_k, t)$ (see Lemma \ref{BazhEquation}).
Using condition
(\ref{wkc}) we obtain the following equation for unknown numbers $h_k$:
\begin{equation}\label{bkeq}
	h_k B_\alpha (\lambda_k, t_0)=\beta h_k.
\end{equation}

Let $\beta \notin (0,1)$ or $\beta \in (0,1)$, but
$K_0$ is an empty set. Then $B_\alpha (\lambda_k, t)\neq \beta$ for all $k$. Consequently, in
this case all $h_k$ are equal to zero.  Ttherefore $w_k(t)=0$, and
by virtue of completeness of the set of eigenfunctions $\{v_k\}$,
we conclude that $w(t)\equiv 0$. Thus, problem (\ref{prob1b}) in
this case has a unique solution.

Now suppose that $\beta \in (0,1)$ and $K_0$ is not empty. Then $B_\alpha (\lambda_k, t)= \beta$, $k\in
K_0$ and therefore equation (\ref{bkeq}) has the following
solution: $h_k=0$ if $k\notin K_0$ and $h_k$ is an arbitrary
number for $k\in K_0$. Thus, in this case, there is no uniqueness
of the solution to problem (\ref{prob1b}). We note that the non-local problem under consideration has a finite-dimensional kernel
\[
Ker = \{h \in H: h=\sum_{k \in K_0}h_k v_k\}.
\]
in this case.

Thus we obtain the following statement:

\begin{theorem}\label{w} Let  $ \psi \in H $.
	If $\beta \notin [0,1)$ or $\beta \in (0,1)$, but $K_0$ is empty, then problem (\ref{prob1b}) has a
	unique solution and this solution has the form (\ref{w>1}).
	If $\beta \in (0,1)$ and $K_0$ is not empty,
	then a solution to problem (\ref{prob1b}) exists if and only if the orthogonality conditions (\ref{orto}) are
	satisfied.  The solution of problem (\ref{prob1b}) has the form
	(\ref{U1}) with arbitrary coefficients $h_k$, $k\in K_0$.
	Moreover, there is a constant $ C_\beta> 0 $ such that the
	following coercive
	estimate holds:
	\begin{equation}\label{west}
		||\partial_t w(t)||^2 + ||Aw(t)||^2 + ||\partial_t^\alpha Aw(t)||^2\leq C_{\beta} t^{-2(1-\alpha)}
		||\psi||^2,\quad 0 < t \leq T.
	\end{equation}
\end{theorem}
\noindent
Note that the proof of the coercive inequality (\ref{west})
follows from the estimates (\ref{Aw}), (\ref{Dalphaw}) and (\ref{Dw}).

Now we are ready to solve the main problem  in (\ref{probMain}). Let
$ \varphi \in H $ and $f(t) \in C ([0,T];D(A^\varepsilon))$ for
some $\varepsilon\in (0,1)$.  If we put
$\psi=\varphi-\omega(t_0)\in H$ and $\omega(t)$ and $w(t)$ are the
corresponding solutions of problems (\ref{prob1a}) and
(\ref{prob1b}), then the function $u(t)=\omega(t)+w(t)$ is a solution
to problem (\ref{probMain}). Therefore, if $\beta \notin (0,1)$ or
$\beta \in (0,1)$, but $K_0$ is empty, then
\begin{equation}\label{u}
	u(t)= \sum\limits_{k=1}^\infty \left[\frac{\varphi_k-\omega_k
		(t_0)}{B_{\alpha}(\lambda_k, t_0)-\beta}   \,
	B_{\alpha}(\lambda_k, t)+\omega_k(t)\right] v_k,
\end{equation}
where
$$
\omega_k(t)=\int\limits_{0}^tB_{\alpha}(\lambda_k, \eta)f_k(t-\eta)d\eta.
$$
The uniqueness of the solution $u(t)$ of problem  \eqref{probMain}  follows from the uniqueness
of the solutions $\omega(t)$ and $w(t)$.

If $\beta \in (0,1)$ and $K_0$ is not empty, then
\begin{equation}\label{u1}
	u(t)= \sum\limits_{k\notin K_0} \left[\frac{\varphi_k-\omega_k
		(t_0)}{B_{\alpha}(\lambda_k, t_0)-\beta}   \,
	B_{\alpha}(\lambda_k, t)+\omega_k(t)\right] v_k
	+
	\sum\limits_{k\in K_0} h_k B_{\alpha}(\lambda_k, t)v_k,
\end{equation}
where $h_k$ are arbitrary numbers.
The corresponding orthogonality conditions have the form
\begin{equation}\label{Orto}
	(\varphi,v_k)=(\omega(t_0), v_k), \,\, k\in K_0.
\end{equation}
In particular, if
\begin{equation}\label{Orto1}
	(\varphi,v_k)=0, \,\, (f(t), v_k)=0,\,\, \text{for all}\,\, 0\leq t\leq t_0,
	\,\, k\in K_0,
\end{equation}
then the orthogonality conditions (\ref{Orto}) are satisfied.

Thus we proved the following statement.

\begin{theorem}\label{1main} Let  $ \varphi \in H $ and $f(t) \in C ([0,T];D(A^\varepsilon))$ for some $\varepsilon\in (0,1)$.
	If $\beta \notin [0,1)$ or $\beta \in (0,1)$, but $K_0$ is empty, then problem (\ref{probMain}) has a
	unique solution and this solution has the form (\ref{u}).
	If $\beta \in (0,1)$ and $K_0$ is not empty,
	then we assume that the orthogonality conditions (\ref{Orto1}) are
	satisfied.  The solution of problem (\ref{probMain}) has the form
	(\ref{u1}) with arbitrary coefficients $h_k$, $k\in K_0$.
	Moreover, there are constants $ C_\beta> 0 $ and
	$C_\varepsilon>0$ such that the following coercive estimate
	holds:
	\[
	||\partial_t w(t)||^2 + ||Aw(t)||^2 + ||\partial_t^\alpha Aw(t)||^2\leq C_\beta t^{-2(2-\alpha)}
	||\varphi||^2 + C_\varepsilon\max\limits_{t\in[0,T]}
	||f||_{\varepsilon}^2,
	\]
	\[
	0 < t \leq T.
	\]
\end{theorem}

Note that the validity of the assertions in Theorem \ref{w} and Theorem \ref{1main} requires that the orthogonality conditions (\ref{orto}) and (\ref{Orto1}) be satisfied, respectively. In light of these conditions a natural question arises: how restrictive are these orthogonality conditions. To answer this question, consider the following example.

Let a bounded domain
$\Omega\subset \mathbb{R}^N$ have sufficiently smooth
boundary $\partial \Omega$. Consider the operator $A_0$, defined in $L_2(\Omega)$ with domain of definition $D(A_0)=\{f\in
C^2(\Omega)\cap C(\overline{\Omega}):\, f(x)=0,\, x\in \partial
\Omega\}$ and acting as $A_0f(x)=-\triangle f(x)$. Then, as is known (see, e.g.
\cite{Il}), $A_0$ has a complete in $L_2(\Omega)$ system of
orthonormal eigenfunctions $\{v_k(x)\}$ and a countable set of
nonnegative eigenvalues $\lambda_k$ ($\rightarrow +\infty$), and
$\lambda_1 = \lambda_1(\Omega)>0$.

Denote by $A$ the operator, acting as $Af(x) =\sum \lambda_k
f_k v_k(x)$ with the domain of definition $D(A)=\{ f\in
L_2(\Omega): \,\sum \lambda^2_k f^2_k<\infty\}$. Then it is not
hard to verify, that $A$ is a positive self-adjoint extension in
$L_2(\Omega)$ of operator $A_0$. Therefore, one can apply Theorems
\ref{w}
and \ref{1main} to the problem:
\begin{equation}\label{prob3}
	\left\{
	\begin{aligned}
		&\partial_t w(x,t) -(1+\gamma \partial_t^\alpha)\triangle w(x,t) = 0, \quad x\in
		\Omega, \quad
		0<t\leq T, \,\, 0<\alpha<1;\\
		&w(x,t_0) =\beta w(x,0) +\psi(x), \,   \, x\in
		\Omega,
		\quad 0<t_0\leq T;\\
		&w(x,t) = 0, \quad  x\in
		\partial \Omega, \quad 0<t\leq T,
	\end{aligned}
	\right.
\end{equation}

Suppose $\beta\in (0,1)$ and $t_0\in (0, T]$ satisfies conditions of Lemma \ref{mainT0}. Then, according to Lemma \ref{mainT0}, only one eigenvalue can satisfy equation (\ref{eqbeta}). Let this number be $\lambda_1$, i.e.
$
B_\alpha(\lambda_1, t_0)=\beta.
$
We note also that the multiplicity $\lambda_1$ is equal to one.

Therefore, applying Theorem \ref{w} we have that
problem (\ref{prob3}) has a solution for any function $\psi\in L_2(\Omega)$, if and only if
\[
\psi_1=\int\limits_\Omega \psi(x) v_1(x) dx=0.
\]
In other words the first Fourier coefficient  of $\psi(x)$ must be zero.
In this case, the solution of the problem is not unique and all solutions can be represented in the series form
\[
w(t)=\sum\limits_{k=2}^\infty \frac{\psi_k}{B_\alpha (\lambda_k, t_0)-\beta} B_\alpha (\lambda_k, t)v_k+ h B_\alpha (\lambda_1, t)v_1,
\]
that converges in the norm of $L_2(\Omega)$  uniformly in $t\in [0,T].$ Here $h$ is an arbitrary real number.

\section{Conclusion}

In this paper, for the Rayleigh-Stokes equation, we study a new time-nonlocal problem, i.e., in the problem (\ref{probIN}), instead of the initial condition $u(x,0)=\varphi(x )$, we consider the nonlocal condition $ u(x, t_0)=\beta u(x, 0) +\varphi(x)$, $0< t_0\leq T$. Instead of the Laplace operator $(-\Delta)$, we take an arbitrary positive self-adjoint operator $A$, and the naturally obtained results are also valid for an equation with a Laplace operator with the Dirichlet condition.

Cases $\beta=0$ and $\beta=1$ were studied earlier: if $\beta=0$, then we get a well-known time backward problem that has a unique solution, but the solution is not stable. If $\beta=1$, then the problem becomes "good", i.e. there is a unique solution and it is stable (see \cite{AshurovVaisova}).

Naturally, the question arises: starting from what value $\beta$ does the task deteriorate? This paper provides a comprehensive answer to this question. It turns out that the critical values of the parameter $\beta$ lie on the half-interval $[0,1)$. If $\beta\notin [0,1)$, then the problem is well-posed according to Hadamard: there is a unique solution and it continuously depends on the data of the problem; if $\beta\in (0,1)$ (case $\beta=0$ is considered in \cite{AshurovVaisova}), then the well-posedness of the problem depends on the location of the eigenvalues of the Laplace operator, namely, if the set $K_0$, defined above, is empty, then the problem is again well-posed according to Hadamard . If $K_0$ is not empty, then necessary and sufficient conditions are found in the paper, the fulfillment of which guarantees the existence of a solution, but in this case the solution is not unique.


\begin{thebibliography}{99}
 \normalsize

%\bibliographystyle{amsplain}
%\begin{thebibliography}{10}

\bibitem{KilSriTru} {\sc A.~Kilbas, H.~Srivastava, J.~Trujillo,} {\it Theory and Applications of Fractional Differential Equations.} Elsevier, Amsterdam, 2006.
	
\bibitem{Kulish} {\sc V.~Kulish, J.~Luis Lage,} {\it Application of fractional calculus to fluid mechanics,} J. Fluids Engg. 124, No 3, 413-442, 2002.
	
\bibitem{Debnath} {\sc L.~Debnath,} {\it Recent applications of fractional calculus to science and engineering,}  Int. J. Math. Sci. 203, No 54, 3413-3442, 2003.
	
\bibitem{Bazh} {\sc E.~Bazhlekova, B.~Jin, R.~Lazarov, Z.~Zhou,} {\it An analysis of the Rayleigh-Stokes problem for a generalized second-grade
fluid.}  Numer. Math. 131, 1-31, 2015.

\bibitem{Chud} {\sc A.~
F.~Chudnovsky,}  {\it Thermal physics of soils.} Nauka, 1976. [In Russian].
	
\bibitem{Nakh1} {\sc A.~M.~Nakhushev,}  {\it Problems with displacement for partial differential equations.} Nauka, 2006. [In Russian].
	
\bibitem{Nakh2} {\sc A.~M.~Nakhushev,}  {\it Loaded equations and their application.} Nauka, 2012. [In Russian].
	
\bibitem{Tan1}	{\sc W.~C.~Tan, T.~Masuoka,} {\it Stokes' first problem for a second grade fluid in a porous half-space with heated boundary.}  Int. J. Non-Linear Mech., 40, 515-522, 2005.

\bibitem{Tan2} {\sc W.~C.~Tan, T.~ Masuoka,} {\it Stokes' first problem for an Oldroyd-B fluid in a porous half-space.}  Phys. Fluid, 17,  023101-7, 2005.
	
\bibitem{Fet} {\sc C.~Fetecau, M.~Jamil, C.~Fetecau, D.~Vieru,} {\it The Rayleigh-Stokes problem for an edge in a generalized Oldroyd-B fluid.} Z. Angew. Math. Phys., 60, No 5, 921-933, 2009.
	
	
\bibitem{Shen} {\sc F.~Shen, W.~Tan, Y.~Zhao, T.~Masuoka,} {\it The Rayleigh-Stokes problem for a heated generalized second grade fluid with fractional derivative model.} Nonlinear Anal. Real World Appl., 7, No 5,  1072-1080, 2006.
	
	
\bibitem{Zhao} {\sc C.~Zhao, C.~Yang,} {\it Exact solutions for electro-osmotic flow of viscoelastic fluids in rectangular
	micro-channels.}  Appl. Math. Comp., 211, No 2, 502-509, 2009.
	
	
\bibitem{Le} {\sc L.~D.~Long, B.~Moradi, O.~Nikan, Z.~Avazzadeh, A.~M.~Lopes,}  {\it Numerical Approximation of the Fractional Rayleigh-Stokes Problem Arising in a Generalised Maxwell Fluid.}  Fractal Fract., 6,  377, 2022.
	
	
\bibitem{Dai} {\sc D.~D.~Dai, T.~T.~Ban,  Y.~L.~Wang,  W.~Zhang,  T.~Hang,}  {\it The piecewise reproducing kernel method for the time variable fractional order
	advection-reaction-diffusion equations.} Thermal science, 25, 1261-1268, 2021.
	
\bibitem{Tran1} {\sc T.~T.~Binh,  D.~Baleanu, N.~H.~Luc, N.~Can,} {\it Determination of source term for the fractional Rayleigh-Stokes equation with
	random data.} Journal of Inequalities and Applications, 308,  1-16, 2019.
	
\bibitem{Tran2} {\sc T.~T.~Binh, H.~K.~Nashine,  L.~D.~Long, N.~H.~Luc,  N.~Can,} {\it Identification of source term for the ill-posed
	Rayleigh-Stokes problem by Tikhonov
	regularization method.} Advances in Difference Equations,  331, 1-20, 2019.
	
	
\bibitem{Duc} {\sc P.~N.~Duc, H.~D.~Binh,  L.~D.~Long,  T.~V.~Kim,}
	{\it Reconstructing the right-hand side of the
	Rayleigh-Stokes problem with non-local in
	time condition.}  Advances in Difference Equations, 470,  1-18, 2021.
	
\bibitem{AshurovVaisova} {\sc R.~Ashurov, N.~Vaisova,} {\it Backward and Non-Local Problems for the Rayleigh-Stokes Equation.}  Fractal Fract., 6, No 10,  587, 2022.
	
\bibitem{Kirane_f_1} {\sc M.~Kirane, A.~S.~Malik,}   {\it Determination of an unknown source term and the temperature distribution for the linear heat equation involving fractional derivative in time.}  Applied Mathematics and Computation, 218, 163-170, 2011.
	
\bibitem{Kirane_f_2} {\sc M.~Kirane, B.~Samet, B.~T.~Torebek,}  {\it Determination of an unknown source term and the temperature distribution for the subdiffusion equation at the initial and final data.} Electronic Journal of Differential Equations, 217, 1-13, 2017.
	
\bibitem{AshMuk} {\sc R.~R.~Ashurov, A.~T.~Mukhiddinova,}  {\it Inverse problem of determining the heat source density for the subdiffusion equation.} Differential Equations, 56, 1550-1563, 2020.
	
	
	
\bibitem{Luc1} {\sc H.~L.~Nguyen,  H.~T.~Nguyen, K.~Mokhtar, X.~T.~Duong Dang,}
	{\it Identifying initial condition of the Rayleigh-Stokes problem
	with random noise.} Wiley, 6,  1-11, 2018.
	
	
	
\bibitem{Luc2} {\sc H.~L.~ Nguyen, L.~N.~Huynh, O.~D.~Regan,  N.~H.~Can,}  {\it Regularization of the fractional Rayleigh-Stokes equation using a fractional
	Landweber method.} Advances in Difference Equations,  459,  1-23, 2020.
	
\bibitem{Ashyr} {\sc A.~O.~ Ashyralyev,}  {\it Nonlocal boundary-value problems for abstract parabolic equations: well-posedness in Bochner spaces.}  Journal of Evolution Equations, 6, 1-28, 2006.
\bibitem{AshyrSob} {\sc A.~O.~ Ashyralyev,   A.~ Hanalyev, P.~E.~Sobolevskii,} {\it Coercive solvability of nonlocal boundary value problem for parabolic equations.}  Abstract and Applied Analysis,  6, 53-61, 2001.
	
\bibitem{AshyrSob1} {\sc A.~O.~ Ashyralyev,  P.~E.~Sobolevskii,} {\it Coercive stability of a multidimensional difference elliptic equation of 2m-th order with variable coefficients.}  Investigations in the Theory of Differential Equations,  Minvuz Turkmen. SSR, Ashkhabad, 31-43, 1987. [in Russian].
	
	
\bibitem{AF2022} {\sc R.~R.~Ashurov, Yu.~Fayziev,} {\it On the nonlocal problems in time for time-fractional subdiffusion equations.} Fractal and Fractional , 6, No 41, 1-21, 2022.
	
\bibitem{AF_1_2022} {\sc R.~R.~ Ashurov, Yu.~Fayziev,}  {\it On the nonlocal problems in time for subdiffusion equations with the Riemann-Liouville derivatives.}  Bulletin of the Karaganda University, 106, No 2, 18-37, 2022.
	
\bibitem{Tran} {\sc T.~T.~Phong,  L.~D.~Long,} {\it  Well-posed results for nonlocal fractional parabolic equation involving Caputo-Fabrizio operator.}  J. Math. Computer Sci., 26, 357-367,2022.
	
\bibitem{Liz} {\sc C.~Lizama,} {\it Abstract linear fractional evolution equations.}  Handbook of Fractional Calculus with Applications, 2, 465-497, 2019.
	
\bibitem{Pskhu} {\sc A.~V.~Pskhu,}  {\it Initial problem for linear ordinary differential
	fractional order equations.} Mat. Sb.,  4,  111-122, 2011. [In Russian].
	
	
\bibitem{Il}{\sc V.~A.~Il'in,} {\it On the solvability of mixed problems for hyperbolic and parabolic equations,}  Math. Surveys, 15, 85-142, 1960.[in Russian] 
	
	


\end{thebibliography}
\end{document}